\documentclass[a4paper,11pt]{amsart}
\usepackage{amsmath,amssymb,url,comment, enumerate}

\newcommand{\Z}{\mathbb{Z}}
\newcommand{\N}{\mathbb{N}}
\newcommand{\Q}{\mathbb{Q}}
\newcommand{\Hy}{\mathbb{H}}
\newcommand{\A}{\mathbb{A}}

\newcommand{\calP}{\mathcal{P}}

\newcommand{\ep}{\varepsilon}
\newcommand{\sca}{\mathfrak{s}}
\newcommand{\nm}{\mathfrak{n}}
\newcommand{\vol}{\mathfrak{v}}
\newcommand{\qf}[1]{\langle #1 \rangle}

\newcommand{\ord}{\operatorname{ord}}
\newcommand{\modfour}{\,(\text{mod}\,4\Z_2)}
\newcommand{\modeight}{\,(\text{mod}\,8\Z_2)}
\newcommand{\prepd}{\overset{*}{\rightarrow}}
\newcommand{\repd}{\rightarrow}
\newcommand{\prin}{\overset{*}{\in}}
\newcommand{\nrepd}{\not \rightarrow}
\newcommand{\lan}{\langle}
\newcommand{\ran}{\rangle}
\newcommand{\rk}[1]{\text{rk}\,#1}

\newtheorem{theorem}{Theorem}[section]
\newtheorem{lemma}[theorem]{Lemma}
\newtheorem{corollary}[theorem]{Corollary}
\newtheorem{proposition}[theorem]{Proposition}

\theoremstyle{remark}
\newtheorem{remark}[theorem]{Remark}
\newtheorem{example}[theorem]{Example}

\numberwithin{equation}{section}

\begin{document}

\title[On locally primitively universal quadratic forms]{On locally primitively universal quadratic forms}

\author[A.G. Earnest]{A.G. Earnest}
\address{Department of Mathematics, Southern Illinois University, Carbondale, IL, 62901, U.S.A.}
\email{aearnest@siu.edu}

\author[B.L.K. Gunawardana]{B.L.K. Gunawardana}
\address{Department of Mathematics, Southern Illinois University, Carbondale, IL, 62901, U.S.A.}
\email{laksh.gg@siu.edu}


\maketitle

\begin{abstract}
A positive definite integral quadratic form is said to be almost (primitively) universal if it (primitively) represents all but at most finitely many positive integers. In general, almost primitive universality is a stronger property than almost universality. The two main results of this paper are: 1) every primitively universal form nontrivially represents zero over every ring $\Z_p$ of $p$-adic integers, and 2) every almost universal form in five or more variables is almost primitively universal.
\end{abstract}

\section{Introduction}

\subsection{Background and history}
A positive definite integral quadratic form $f$ is said to be \textit{almost universal} if it represents all sufficiently large positive integers; that is, if the excluded set of positive integers not represented by $f$ is finite. The systematic study of forms with this property was initiated by Ramanujan over a century ago. In a groundbreaking 1917 paper \cite{R}, he determined all diagonal quaternary forms of the special type $ax^2+ay^2+az^2+dt^2$ which have this property. Among them is the form $x^2+y^2+z^2+9t^2$ which represents all positive integers with the single exception of the number $7$. Halmos \cite{H} subsequently determined all diagonal quaternary forms for which the excluded set consists of a single integer. There are 88 such forms $ax^2+by^2+cz^2+dt^2$ with $a\leq b \leq c\leq d$.

More recently, Bochnak and Oh \cite{BO} completed the determination of effective criteria whereby it can be decided whether or not a general (not necessarily diagonal) positive definite quaternary integral quadratic form $f$ is almost universal. It can be seen that a necessary condition for $f$ to be almost universal is that it be everywhere locally universal; that is, that for every prime $p$, the equation $f(x_1,\ldots,x_4)=a$ is solvable over the ring $\Z_p$ of $p$-adic integers for every $a\in \Z_p$. Indeed, if this local condition does not hold for some prime $p$, then $f$ fails to represent an entire arithmetic progression of positive integers. However, for quaternary forms $f$, these local conditions are not sufficient to guarantee that $f$ is almost universal. For example, the form $f=x^2+y^2+5^2z^2+5^2t^2$ is everywhere locally universal but fails to represent the integers of the type $3\cdot 2^{2k}$, where $k$ is any nonnegative integer. Note that this $f$ fails to represent zero nontrivially $2$-adically; that is, it is anisotropic over $\Z_2$. Bochnak and Oh refer to almost universal forms that are anisotropic for some prime $p$ as \textit{exceptional} and prove that this terminology is indeed appropriate, in the sense that there are only finitely many equivalence classes of such forms. The 144 diagonal quaternary forms of this type were essentially determined by Kloosterman \cite{K}.

In light of classical theorems of Tartakowsky \cite{T} and Ross and Pall \cite{RP} (see \cite[Theorem 1.6, page 204]{Ca}), the stronger condition that $f$ \underline{primitively} represents all integers over $\Z_p$ for all primes $p$ \underline{does} imply that $f$ is almost universal; in fact, such a form \underline{primitively} represents all sufficiently large integers. A positive definite integral quadratic form $f$ is said to be \textit{almost primitively universal} if for all sufficiently large positive integers $a$, there exist $x_1,\ldots,x_n \in \Z$ such that $f(x_1,\ldots,x_n)=a$ and $\text{g.c.d.}\,(x_1,\ldots,x_n)=1$. While it is clear from the definitions that almost primitive universality implies almost universality, the converse is not true. For example, the form $f=x^2+y^2+z^2+9t^2$ mentioned in the first paragraph does not primitively represent the number $8$ over $\Z_2$, and consequently does not primitively represent any member of the arithmetic progression $8+64k$ over $\Z$. So this form $f$ is almost universal, but not almost primitively universal.

As noted above, the identification of almost primitively universal forms reduces completely to a local problem; that is, a positive definite integral quadratic form is almost primitively universal if and only if it is everywhere locally primitively universal. In the paper \cite{nB}, Budarina began a detailed investigation of the local problem of determining when a form over $\Z_p$ is primitively universal, and used these results to derive some criteria for a form of odd discriminant to be almost primitively universal. For example, in the spirit of the celebrated fifteen theorem of Conway and Schneeberger (see \cite{Co}, \cite{mB}), Budarina proved: {\em a classically integral quadratic form in at least four variables with odd squarefree discriminant is almost primitively universal if and only if it primitively represents the integers $1$, $4$ and $8$}. In the present paper, we will extend this line of inquiry, revisiting and generalizing the local investigations of Budarina. In the process, we produce new, more elementary proofs of the main results in the paper \cite{nB}, and extend the results there to the case of forms of even discriminant. In contrast to Budarina's arguments, which draw heavily on Zhuravlev's extensive general work on minimal indecomposable representations, as described in \cite{Z} and the references therein, our methods make use of nothing more advanced than standard local theory as presented, for example, in the foundational books of O'Meara \cite{OM} and Gerstein \cite{G}.

The results in this paper will form the basis for many of the arguments in a forthcoming paper in which the authors will give explicit criteria for local primitive universality of an integral quadratic form of any rank. The results obtained there will be further applied in that paper to complete the work initiated by Budarina \cite{nB} of determining which among the universal positive definite classically integral quaternary quadratic forms are almost primitively universal.

\subsection{Statement of main results} We will state here the main results of this paper in the traditional language of quadratic forms, although the proofs will subsequently by presented from the more modern geometric perspective of quadratic lattices. The first two theorems hold for arbitrary integral quadratic forms $f = \sum_{1\leq i \leq j \leq n} a_{ij}X_iX_j$ with $a_{ij} \in \Z$, with no additional assumption on the cross-term coefficients $a_{ij}$, $i\neq j$.

\begin{theorem}\label{Theorem 1}
If a positive definite integral quadratic form is almost primitively universal, then it non-trivially represents zero over $\Z_p$ for all primes $p$.
\end{theorem}

\noindent Thus, no almost universal forms of the exceptional type of Bochnak and Oh can be almost primitively universal. The next result shows that the distinction between almost universality and almost primitive universality is no longer present when the number of variables exceeds four.

\begin{theorem}\label{Theorem 2}
If a positive definite integral quadratic form in five or more variables is almost universal, then it is almost primitively universal.
\end{theorem}

\noindent As an application of this theorem and our $2$-adic computations, we prove the following result, which extends Theorem 6 and Corollary 2 of \cite{nB}. For this statement, we restrict to forms that are classically integral, in the sense that the cross-term coefficients are even integers.

\begin{theorem}\label{Theorem 3}
Let $f$ be a positive definite classically integral quadratic form in $n$ variables such that $f$ represents an odd integer and the discriminant of $f$ is not divisible by $p^{n-2}$ for any prime $p$. If $n\geq 5$, or if $n=4$ and the discriminant of $f$ is even, then $f$ is almost primitively universal.
\end{theorem}

\noindent Finally, although the main focus of this paper is on positive definite forms, the local computations that yield the proof of Theorem~\ref{Theorem 2} can be applied to arbitrary integral quadratic forms in the indefinite case as well, to obtain the following result:

\begin{theorem}\label{Theorem 4}
Let $f$ be an indefinite integral quadratic form in five or more variables. If every integer is represented by the genus of $f$, then every nonzero integer is primitively represented by $f$.
\end{theorem}

\subsection{Organization}
The remainder of this paper will be organized as follows. In Section 2, some basic definitions are reviewed, terminology and notation is summarized, and several preliminary results are established. As noted above, the proofs of the theorems stated in the previous section can be reduced to local arguments. The local considerations required for the proof of Theorem~\ref{Theorem 1} are given in Section 3, with the main results there being Proposition~\ref{N a.p.u. cond} and its corollary. The proof of Theorem~\ref{Theorem 2} requires detailed computations of the primitive representations of $\Z_p$-lattices. The determination of these representations is separated into two sections; those for odd primes $p$ are given in Section 4, with the main result being Proposition~\ref{non dyadic}, and those for $p=2$ in Section 5, culminating in Proposition~\ref{u implies pu p=2}. Remaining details of the proofs of the theorems are contained in Section 6.

\section{Preliminaries}

\subsection{Notation and terminology for lattices}
For the remainder of this paper we will abandon the language of forms, and instead adopt the geometric language of lattices. Unexplained terminology and notation will follow that presented in the books of O'Meara \cite{OM} and Gerstein \cite{G}. To set the context, let $R$ be an integral domain with field of quotients $F$ of characteristic not 2. By an {\em $R$-lattice} $L$, we will mean a finitely generated $R$-submodule of a nondegenerate quadratic space $V$ over $F$ equipped with a quadratic map $q$ and corresponding symmetric bilinear form $B$ for which $q(v)=B(v,v)$ for all $v \in V$. We will say that the $R$-lattice is {\em integral} if $q(L) \subseteq R$, where $q(L)= \{q(v): v \in L \}$.  A vector $v \in L$ is {\em primitive} in $L$, denoted $v \overset{*}{\in}L$, if $\{\alpha \in F: \alpha v \in L  \}  =R$. An element $a \in F$ is said to be {\em represented} ({\em primitively represented}, resp.), denoted $a \rightarrow L$ ($a \xrightarrow{*} L$, resp.), if there exists $v \in L$ ($v \overset{*}{\in} L$, resp.) such that $q(v)=a$. For a set $S$ of elements of $F$, the notation $S \rightarrow L$ ($S \xrightarrow{*} L$, resp.) will be used to indicate that $a \rightarrow L$ ($a \xrightarrow{*} L$, resp.) for all $a \in S.$ The notation $S \not\repd L$ will mean that there exists at least one element $a\in S$ such that $a\not\repd L$, and analogously for $S\not\prepd L$.

For our purposes, the ring $R$ will always be either the ring $\Z$ of rational integers or a ring $\Z_p$ of $p$-adic integers for some prime $p$. Since these rings are principal ideal domains, all lattices under consideration will be free. If $\mathcal{B} = \{v_1, \ldots , v_n\}$ is a basis for a lattice $L$, the matrix $(B(v_i, v_j))$ is the {\em Gram matrix of $L$ with respect to $\mathcal{B}$.} For a symmetric $n \times n$-matrix $M$, we will write $L \cong M$ to indicate that there exists a basis for $L$ such that $M$ is the Gram matrix of $L$ with respect to that basis. In particular, $L \cong \langle a_1, \ldots , a_n \rangle$ will mean that $L$ has an orthogonal basis for which the Gram matrix is the diagonal matrix with the indicated diagonal entries. When $L$ is a $\Z$-lattice, all Gram matrices of $L$ have the same determinant, which is called the {\em discriminant of $L$} and denoted $dL$. When $L$ is a $\Z_p$-lattice, the determinants of all Gram matrices of $L$ lie in the same coset of $\dot\Q_p/(\Z_p^{\times})^2$, where $\Z_p^{\times}$ denotes the group of units of $\Z_p$. This coset is the {\em discriminant of $L$}, again denoted by $dL$. The notation $dL$ is also used to denote the determinant of a specific Gram matrix of $L$, with the exact meaning generally clear from the context.

\subsection{The $p$-adic integers}
In this subsection we will review a few facts regarding the $p$-adic integers that will be used frequently in the remainder of the paper. Further discussion of the topics here and in the next subsection can be found, for example, in the books of O'Meara \cite{OM} or Gerstein \cite{G}. For a prime $p$, $\Q_p$ will denote the field of $p$-adic numbers (that is, the completion of the rational number field $\Q$ with respect to the $p$-adic metric $|\cdot|_p$) and $\Z_p$ will denote the ring of integers of $\Q_p$ (so $\Z_p = \{ \alpha \in \Q_p : |\alpha|_p \leq 1\}$). This ring is a local ring with unique maximal ideal $p\Z_p$, and all fractional ideals of $\Z_p$ in $\Q_p$ are of the form $(p\Z_p)^j$ for some $j \in \Z$; hence the nonzero fractional ideals of $\Z_p$ in $\Q_p$ are linearly ordered by inclusion.

The group of units of $\Z_p$ will be denoted by $\Z^{\times}_p$; thus, $\Z^{\times}_p = \{\alpha \in \Z_p : |\alpha|_p = 1 \} $. So a typical element $a\in \Z_p$ can be written uniquely as $a=p^{\ord_pa}a_0$ with $\ord_pa \in \N \cup\{0\}$ and $a_0\in \Z_p^{\times}$. We will make frequent use of the Local Square Theorem, which in the present context asserts that for $\alpha, \beta \in \Z^{\times}_p$, if $\alpha \equiv \beta (\textup{mod} \ 4p\Z_p)$ then $\alpha \in \beta(\Z^{\times}_p)^2$ (e.g., see \cite[Theorem 3.39]{G}). From this it follows that when $p$ is odd, the group $\Z^{\times}_p $ consists of two squareclasses; that is, $\Z^{\times}_p = (\Z^{\times}_p)^2 \cup \Delta(\Z^{\times}_p)^2$, where $\Delta$ denotes any fixed nonsquare in $\Z^{\times}_p$. For $p=2$, the group $\Z^{\times}_2 $ consists of four squareclasses with representatives 1, 3, 5 and 7. If $\alpha, \beta \in \Z^{\times}_2$, then $\alpha \in (\Z^{\times}_2)^2$ if and only if $\alpha \equiv 1 (\textup{mod} \ 8\Z_2)$, and  $\alpha(\Z^{\times}_2)^2 = \beta(\Z^{\times}_2)^2 $ if and only if $\alpha \equiv \beta (\textup{mod} \ 8\Z_2)$.

\subsection{Orthogonal splittings and invariants of lattices over $\Z_p$ }
In the case of lattices over $\Z_p$, there are several invariants that will frequently be used in the arguments that follow. For such a lattice $L$, we let $\mathfrak{s}L$, $\mathfrak{n}L$ and $\mathfrak{v}L$ denote the scale, norm and volume ideals, respectively, associated to the lattice $L$, as defined in \cite[\S 82E]{OM}. When a lattice can be decomposed as an orthogonal sum of sublattices whose norm ideals are distinct, it is sometimes possible to transfer information on representations between the entire lattice and the sublattices. For example, we will make frequent use of the following results:

\begin{lemma}\label{rep of units}
Let $L$ be an integral $\Z_p$-lattice such that $L \cong M \perp K$ for nonzero sublattices $M$ and $K$ of $L$.

i) If $M \cong \langle \varepsilon \rangle$, for some $\varepsilon \in \Z^{\times}_p $, and $\nm K \subseteq2p\Z_p$, then $\Z^{\times}_p \not \rightarrow L$.

ii) If $ \Z^{\times}_p  \rightarrow L$ and $\nm K \subseteq 4p\Z_p$, then $\Z^{\times}_p  \rightarrow M$.\\

iii) If $ p\Z^{\times}_p  \rightarrow L$ and $\nm K \subseteq 4p^2 \Z_p$, then $p\Z^{\times}_p  \rightarrow M$.
\end{lemma}

\begin{proof}
$i)$ When $p$ is odd, it follows from the Local Square Theorem that the only units that are represented by $L$ are in $\varepsilon(\Z^{\times}_p)^2$. For $p=2$, any unit represented by $L$ is congruent to $\varepsilon$ modulo $4\Z_2$. Hence at most two squareclasses of units can be represented by $L$.

$ii)$ Let $\mu $ be a unit represented by $M \perp K.$ Then $\mu = q(x) + q(y)$ for some $x $ in $ M $ and $y $ in $ K.$ Since $\nm K \subseteq 4p\Z_p$, we have $\mu $ congruent to $q(x) $ modulo $ 4p\Z_p$. So there exists $\lambda \in\Z_p^{\times}$ such that $\mu = \lambda^2q(x)=q(\lambda x)\in q(M)$.

$iii)$ Similar argument as $ii)$.
\end{proof}

The norm and scale of a lattice are related by the following containments:
$$ 2\sca L \subseteq\nm L \subseteq \sca L.$$
In particular, when $p$ is odd it is always the case that $\nm L = \sca L$, and when $p=2$ there are two possibilities, namely $\nm L = \sca L$ or $\nm L = 2\sca L$. The scale and volume of a lattice are related by the containment $\vol L \subseteq (\sca L)^n$, where $n$ is the rank of $L$. When equality holds, the lattice is said to be $\sca L$-{\em modular}. A $\Z_p$-modular lattice is called {\em unimodular}, and a lattice is referred to simply as {\em modular} if it is $A$-modular for some fractional ideal $A$. Thus, a $\Z_p$-lattice $L$ is $p^a\Z_p$-modular if and only if the scaled lattice $L^{a^{-1}}$ is unimodular (here the notation $L^{\alpha}$ denotes the scaled lattice, as defined in \cite[\S 82J]{OM}). For an odd prime $p$, a modular lattice over $\Z_p$ can always be written as an orthogonal sum of rank 1 sublattices (that is, such a lattice is diagonalizable); a modular lattice over $\Z_2$ can be written as an orthogonal sum of modular sublattices of rank 1 or 2 \cite[93:15]{OM}. The modular lattices $L$ over $\Z_2$ which are diagonalizable are precisely those for which $\nm L = \sca L$; these are referred to as {\em proper}. In the case of an improper unimodular lattice $L$ over $\Z_2$, one can more precisely say that $L$ has an orthogonal splitting
\begin{equation}\label{improper unimodular}
L \cong \mathbb{H} \perp \ldots \perp \mathbb{H} \perp P,
\end{equation}
where $\mathbb{H}$ denotes a hyperbolic plane with matrix  $ \left(\begin{smallmatrix} 0 & 1\\ 1 & 0 \end{smallmatrix}\right)$ and $P$ is a binary lattice isometric to either $\mathbb{H}$ or the lattice $\mathbb{A}$ with matrix   $ \left(\begin{smallmatrix} 2 & 1\\ 1 & 2 \end{smallmatrix}\right)$  (e.g., see \cite[Corollary 8.10]{G}).

An arbitrary lattice over $\Z_p$ can always be decomposed as an orthogonal sum of modular sublattices. By grouping the sublattices having the same scale, one obtains the so-called {\em Jordan splitting} for the lattice. For our purposes, we will be considering only $\Z_p$-lattices $L$ for which $\nm L = \Z_p$. Thus $\sca L$ will be either $\Z_p$ or $ \frac{1}{2} \Z_p$ by the fundamental containment noted above. The Jordan splitting of such a lattice can thus be written as
\begin{equation}\label{localsplitting}
     L \cong L_{(-1)} \perp  L_{(0)} \perp L_{(1)} \perp \ldots \perp L_{(t)},
\end{equation}
where each Jordan component $L_{(i)}$ is either $p^i\Z_p$-modular or 0.\footnote{Note that our convention for the indexing of the components differs from that of \cite{OM}.} Here $L_{(-1)}=0$ unless $p=2$. The existence of Jordan splittings and the extent to which such splittings are unique are discussed in detail in \cite[\S 91C]{OM}. Of relevance for our purposes is the fact that the ranks, norms and scales of the Jordan components are invariants of the lattice. The rank of the component $L_{(i)}$ will be denoted by $r_i$. When $L$ is a $\Z_2$-lattice with $\nm L=\Z_2$, the leading Jordan component of $L$ will be either an improper $\frac{1}{2}\Z_2$-modular lattice $L_{(-1)}$ if $\sca L = \frac{1}{2}\Z_2$, or a proper unimodular lattice $L_{(0)}$ if $\sca L=\Z_2$. In order to describe the improper $\frac{1}{2}\Z_2$-modular lattices, it will be convenient to introduce the notations $\widehat{\Hy}$ and $\widehat{\A}$ to represent the lattices obtained from $\Hy$ and $\A$, respectively, by scaling by $\frac{1}{2}$. So \[\widehat{\Hy} \cong \left(\begin{smallmatrix} 0 & \frac{1}{2}\\ \frac{1}{2} & 0 \end{smallmatrix}\right) \text{ \,\,  and \,\,  } \widehat{\A} \cong \left(\begin{smallmatrix} 1 & \frac{1}{2}\\ \frac{1}{2} & 1 \end{smallmatrix}\right).\]

\section{Universality and isotropy of $\Z_p$-lattices}
In this section we will establish some results regarding the set $q^*(L)= \{q(v): v \overset{*}{\in} L \}$ when $L$ is an integral $\Z_p$-lattice. The nature of this set depends heavily on whether the lattice $L$ is isotropic (that is, there exists a $0 \neq v \in L$ such that $q(v) =0$ or, equivalently, $0 \in q^*(L)$) or anisotropic, as the following examples of binary modular lattices illustrate.

\begin{example}\label{3.1}
For any prime $p$, $q^*(\widehat{\Hy})=\Z_p$ (e.g., see \cite[Proposition 3.2]{EKM}).
\end{example}

\begin{example}\label{3.2}
For $p=2$, $q^*(\widehat{\A})=\Z^{\times}_p$. To see this, let $\widehat{\A}$ have the Gram matrix given above with respect to the basis $\{v_1, v_2 \}$; so $q(a_1v_1+a_2v_2)=a^2_1 +a_1a_2+a^2_2$. If one or both of the $a_i$ are in $\Z^{\times}_2$, then $a^2_1 +a_1a_2+a^2_2$ is also in  $\Z^{\times}_2$, thus giving the containment $ q^*(\widehat{\A}) \subseteq \Z^{\times}_2$. To see the reverse containment, it is only necessary to check that the expression $a^2_1 +a_1a_2+a^2_2 $ takes on values from the four squareclasses in $\Z^{\times}_2$. We further note that $q(\widehat{A})=\{\alpha \in \Z_p: \ord_p\alpha \text{ is even}\}$.
\end{example}

\begin{example}\label{3.3}
For any prime $p$, let $L \cong \qf{ \ep_1, \ep_2}$ be anisotropic, where $ \varepsilon_1, \varepsilon_2 \in \Z^{\times}_p $. Then $q^*(L) \cap 4p\Z_p = \emptyset$. To see this, let $\{v_1, v_2 \} $ be the basis for which the Gram matrix is $\qf{\ep_1, \ep_2}$. Suppose that $v=a_1v_1 +a_2v_2 \overset{*}{\in} L$. Without loss of generality, suppose that $ a_1 \in \Z^{\times}_p$. If $q(v) \in 4p\Z_p$, then also $ a_2 \in \Z^{\times}_p$. Then $a^2_1\varepsilon_1 + a^2_2\varepsilon_2 \equiv 0 \ (
\textup{mod} \ 4p\Z_p)$, and it follows that $-\varepsilon_1 \varepsilon_2^{-1}  \equiv (a^{-1}_1a_2)^2   \ (\textup{mod} \ 4p\Z_p) $. By the Local Square Theorem, there exists $\lambda \in \Z^{\times}_p$ such that $-\varepsilon_1 \varepsilon_2^{-1} = \lambda^2$. But then $q(v_1 + \lambda v_2)=0$, contrary to the assumption that $L$ is anisotropic.
\end{example}

An integral $\Z_p$-lattice $L$ will be said to be {\em $\Z_p$-universal} if $q(L) = \Z_p$. Note that $L$ is $\Z_p$-universal if and only if $\Z^{\times}_p \cup p\Z^{\times}_p \subseteq q(L)$. Further, $L$ is said to be {\em primitively $\Z_p$-universal} if for each $0 \neq \alpha \in \Z_p$, there exists $v \overset{*}{\in} L$ such that $q(v)=\alpha$. It is clear that primitive $\Z_p$-universality implies $\Z_p$-universality, but the converse is not true, as can be seen from the following examples.

\begin{example}\label{3.4}
Consider $L \cong \langle 1, 1, 3, 3 \rangle$ over $\Z_3$. It is easily seen that $L$ represents $1, 2, 3$ and $6$, which are representatives of the four squareclasses in $\Z^{\times}_3 \cup 3\Z^{\times}_3$; thus, $L$ is $\Z_3$-universal. On the other hand, write $L = M \perp K$, where $M \cong \lan 1, 1 \ran$ and $K \cong \lan 3,3 \ran$, both of which are anisotropic over $\Z_3$. Let $v \overset{*}{\in} L$. Then $v= x+y$, where $x \overset{*}{\in} M $ or $y \overset{*}{\in} K$. If $x \overset{*}{\in} M $, then $q(x) \in \Z^{\times}_3$ by Example \ref{3.3}, and hence $ q(v) \in 3\Z^{\times}_3$. Otherwise, $x \in 3M$ and $y \overset{*}{\in} K$. In that case, $q(x) \in 9\Z_3$ and $q(y) \in 3\Z^{\times}_3$, again by Example \ref{3.3}, and hence $q(v) \in 3\Z^{\times}_3$. So $q^*(L)= \Z^{\times}_3 \cup 3\Z^{\times}_3$ and $L$ is not primitively $\Z_3$-universal.
\end{example}

\begin{example}
Consider $L \cong \widehat{\A} \perp \A$ over $\Z_2$. Then $L$ is $\Z_2$-universal, but not primitively $\Z_2$-universal and $q^*(L)= \Z^{\times}_2 \cup 2\Z^{\times}_2$. The verifications are as in the previous example.
\end{example}

The following two lemmas will be used frequently in the remainder of this paper.

\begin{lemma}\label{main argument}
Let $L$ be an integral $\Z_p$-lattice such that $L \cong M \perp K$ for nonzero sublattices $M$ and $K$ of $L$. If at least one of $M$ and $K$ is $\Z_p$-universal, then $L$ is primitively $\Z_p$-universal.
\end{lemma}

\begin{proof}
Suppose $M$ is $\Z_p$-universal. For any $\alpha \in \Z_p$ and any $ v \overset{*}{\in}  K$, we have  $\alpha-q(v) \in \Z_p$. So $\alpha-q(v) $ is represented by $M$; hence $\alpha $ is primitively represented by $M \perp K \cong L$.
\end{proof}

\begin{lemma}\label{split by unit}
Let $K$ be an integral $\Z_p$-lattice such that $\Z_p^{\times} \repd K$. Then $\qf{\ep}\perp K$ is primitively $\Z_p$-universal for any $\ep \in \Z_p^{\times}$.
\end{lemma}

\begin{proof} If $\lambda \in p\Z_p$, then $\lambda-\ep \in \Z_p^{\times}$. So $\lambda-\ep \repd K$ and $\lambda\prepd \qf{\ep}\perp K$. 
\end{proof}

We are now ready to state the main result of this section.

\begin{proposition}\label{N a.p.u. cond}
Let $p$ be a prime and $L$ an anisotropic integral $\Z_p$-lattice. Then there exists $l=l(L,p) \in \N $ such that $q^*(L) \cap p^l\Z_p = \emptyset$.
\end{proposition}

Before proceeding to the proof of Proposition~\ref{N a.p.u. cond}, we prove the following lemma.

\begin{lemma}\label{a.p.u. lem}
Let $L \cong M \perp K$ be a $\Z_2$-lattice with $M \cong \mathbb{A}^{2^a}$, $K \cong \lan 2^b \beta \ran \perp \lan 2^c \gamma \ran$, where $a, b, c$ are non-negative integers and $\beta, \gamma \in \Z^{\times}_2$. If $\sca L = \Z_2$ and $L$ is anisotropic, then $a, b, c$ are all even and $\beta + \gamma \equiv 4 (\textup{mod}\,8\Z_2)$.
\end{lemma}

\begin{proof}
If $a, b$ had opposite parity, then, for $k$ such that $2k+b \geq a+1$, we would have $2^{2k+b} \beta \rightarrow K$ and $-2^{2k+b} \beta \rightarrow M$, and it would follow that $L$ is isotropic. So $a,b$ have the same parity, and, similarly, $a, c$ have the same parity. Since $\sca L = \Z_2$, at least one of $a, b, c$ equals $0$, so all of $a, b, c$ must be even. To prove the second assertion, suppose first that $\beta +\gamma \equiv 0(\textup{mod} \ 8\Z_2)$. Then $\beta (\Z^{\times}_2)^2 = -\gamma  (\Z^{\times}_2)^2$ and it would follow that $K$ is isotropic. If $\beta + \gamma \equiv 2(\textup{mod} \ 4\Z_2)$, then $K$ would represent an element of odd order and again $L$ would be isotropic. Since $\beta + \gamma \in 2\Z_2$, this leaves $\beta + \gamma \equiv 4 (\textup{mod} \, 8\Z_2)$ as the only remaining possibility, thus completing the proof. 
\end{proof}

\noindent {\it Proof of Proposition~\ref{N a.p.u. cond}}\, Note that the result for binary lattices is covered by Examples \ref{3.2} and \ref{3.3}, and for ternary lattices the proof is given in \cite[Proposition 3.1]{EKM}. Moreover, the assertion is vacuous when $\rk{L} \geq 5$ since every $\Z_p$-lattice of rank exceeding 4 is isotropic. So we need only consider the case of lattices $L$ of rank 4. Further, by scaling if necessary, we may assume that $\sca L = \Z_p$. So $L$ has a Jordan splitting
$$ L \cong L_{(0)} \perp \ldots \perp L_{(t)},$$
for some non-negative integer $t$, where $L_{(0)} \neq 0$ and $L_{(t)} \neq 0$. Throughout the proof, we let $\{v_1, \ldots , v_4\}$ be a basis for $L$ that gives rise to the indicated Jordan spitting.

Our goal will be to prove that the conclusion of the proposition holds for $l=t+3$, although in some cases a smaller exponent would suffice. So suppose there exists $v \overset{*}{\in} L$ such that $q(v) \in p^{t+3} \Z_p$. Write $v= \sum_{i=1}^{4} b_iv_i$, where $b_1, \ldots , b_4 \in \Z_p$ and $b_k \in \Z^{\times}_p$ for at least one index $k$.

Consider first the case when $v_k$ occurs in the orthogonal basis for a diagonalizable Jordan component of $L$; say $q(v_k)=p^{e_k} \varepsilon_k$, with $\varepsilon_k \in \Z^{\times}_p$. Writing $v= b_k v_k +w$, where $w= \sum_{i\neq k}^{} b_iv_i$, we have
$$b^2_k p^{e_k} \varepsilon_k + q(w)=q(v) \equiv 0 (\textup{mod} \ p^{t+3} \Z_p).$$
It follows that $\ord_pq(w) = e_k$ and
$$ -p^{-e_k} \varepsilon^{-1}_k q(w) \equiv b^2_k (\textup{mod} \ p^{t-e_k +3} \Z_p).$$
Since $t-e_k +3 \geq 3$, it then follows from the Local Square Theorem that there exists $\lambda \in \Z^{\times}_p$ such that
$$-p^{-e_k} \varepsilon^{-1}_k q(w) = \lambda ^2.$$
Then
$$q(\lambda v_k+w)= \lambda ^2 p^{e_k} \varepsilon_k + q(w) = 0,$$
contrary to the assumption that $L$ is anisotropic.

This completes the proof when $p$ is odd, and when $p=2$ and $L$ is diagonalizable. So we need only further consider the case that $p=2$ and $L$ has at least one improper Jordan component. Since $L$ is assumed to be anisotropic, this component must be isometric to $\mathbb{A}^{2^s}$, for some non-negative integer $s$.

Consider first the case that $L \cong \mathbb{A} \perp \mathbb{A}^{2^t}$. If $t$ is even, then $L$ is isotropic; so it suffices to consider odd $t$. Let $v \overset{*}{\in} L$; say $v = x+y$ with $x \in \mathbb{A}$, $y \in \mathbb{A}^{2^t}$. If $x \overset{*}{\in} \mathbb{A}$, then $q(v) \in 2\Z^{\times}_2$. Otherwise, $y \overset{*}{\in} \mathbb{A}^{2^t}$ and $q(y) \in 2^{t+1} Z^{\times}_2$. Since $\ord_2 q(x)$ is odd, $\ord_2 q(x) \neq \ord_2 q(y) = t+1$. So
$$ \ord_2 q(v) = \ord_2 (q(x)+q(y)) = \textup{min} \{\ord_2 q(x), t+1 \} \leq t+1.$$
So we conclude that $q(v) \not \in 2^{t+3} \Z_2$.

In all other cases, the Jordan splitting of $L$ has the form considered in Lemma \ref{a.p.u. lem} for suitable integers $a, b, c$ and units $\beta, \gamma$. By Lemma \ref{a.p.u. lem}, we need only consider the case when $a, b, c$ are all even and $\beta + \gamma \equiv 4(\textup{mod} \ 8\Z_2)$. So $L \cong M \perp K$, with $M \cong \mathbb{A}^{2^a}$, $K \cong \lan 2^b \beta \ran \perp \lan 2^c \gamma \ran $ in the basis $\{u, w\}$. Let $v \overset{*}{\in} L$, and write $v=x+\alpha u + \delta w$, $x \in \mathbb{A}^{2^a}$, $\alpha , \delta \in \Z_2$. Moreover, we may assume that $x \overset{*}{\in} \mathbb{A}^{2^a}$, as the other cases are covered in the first part of the proof. So $q(x) \in 2^{a+1} \Z^{\times}_2$. Write $\alpha = 2^l \alpha_0$, $\delta=2^k \delta_0$, $\alpha_0, \delta_0 \in \Z^{\times}_2$. Then
$$q(\alpha u + \delta w) = 2^{2l+b} \alpha^2_0 \beta + 2^{2k+c} \delta^2_0 \gamma.$$
If $2l+b \neq 2k+c$, then $\ord_2 q(\alpha u + \delta w)= \textup{min} \{2l+b, 2k+c\}$. If $2l+b = 2k+c$, then $\ord_2 q(\alpha u + \delta w)= 2l+b+2$ since $\beta +\gamma \equiv 4 (\textup{mod} \ 8\Z_2)$. In either case, $\ord_2 q(\alpha u + \delta w)$ is even; hence $\ord_2 q(\alpha u + \delta w) \neq \ord_2 q(x) = a+1$, since $a+1$ is odd. So
$$\ord_2 q(v)= \textup{min} \{a+1, \ord_2 q(\alpha u + \delta w)\} \leq t+1.$$
Once again we conclude that $q(v) \not \in 2^{t+3}\Z_2$, and the proof is complete.
\bigskip

\begin{corollary}\label{anisotropic}
If $L$ is an anisotropic integral $\Z_p$-lattice, then $L$ is not primitively $\Z_p$-universal.
\end{corollary}

\smallskip
\section{Primitively universal $\Z_p$-lattices - Non-dyadic case}

Throughout this section, $p$ will denote an odd prime.

\begin{lemma}\label{unimod}
Let $L$ be a unimodular $\Z_p$-lattice.

i) If $\rk{L} = 2$, then $ \Z^{\times}_p \rightarrow L$.

ii) If $\rk{L} \geq 3$, or if $\rk{L} =2$ and $L$ is isotropic, then $L$ is primitively $\Z_p$-universal.
\end{lemma}

\begin{proof}
$i)$ is simply a restatement of \cite[92:1(b)]{OM}. For $ii)$, it suffices to note that either $L \cong \mathbb{H}$ or $L \cong \mathbb{H} \perp \lan 1, \ldots , 1, -d \ran$, where $dL = d(\Z^{\times}_p)^2$, by \cite[92:1]{OM}, and apply the result of Example~\ref{3.1}.
\end{proof}

For an integral $\Z_p$-lattice $L$, we have $\nm L = \sca L \subseteq \Z_p$, and $L$ has a Jordan splitting of the form
$$L \cong L_{(0)} \perp \ldots \perp L_{(t)},$$
where each $L_{i}$ has an orthogonal basis.

\begin{proposition}\label{non dyadic}
Let $p$ be an odd prime and let $L$ be an integral $\Z_p$-lattice with $\rk{L} \geq 5$. If $L$ is $\Z_p$-universal, then $L$ is primitively $\Z_p$- universal.
\end{proposition}

\begin{proof}
Assume that $L$ is $\Z_p$-universal. Then $\nm L = \sca L = \Z_p$ and, by Lemma~\ref{rep of units}$(i)$, $ r_0 \geq 2$. If $L_{(0)}$ is isotropic, then $L$ is primitively $\Z_p$-universal by Lemma~\ref{unimod}$(ii)$. So we need only consider further the case when $r_0 =2$ and $L_{0}$ is anisotropic. In this case, $q^*(L_{(0)}) \cap p\Z_p = \emptyset$ by Example~\ref{3.3}. If $r_1= 0$, it would follow from Lemma~\ref{rep of units}$(iii)$ that $p\Z_p \not \rightarrow L$, contrary to the assumption that $L$ is $\Z_p$-universal. So to complete the proof we consider three possibilities for $r_1=\rk{L_{(1)}}$:

$r_1 =1$: Since $L_{(0)}$ is anisotropic, $ p\Z^{\times}_p \rightarrow L_{(0)} \perp L_{(1)} $ holds if and only if $p\Z^{\times}_p \rightarrow L_{(1)} \perp p L_{(0)}$. But $\Z^{\times}_p \not \rightarrow (L_{(1)} \perp p L_{(0)})^{1/p}$ by Lemma~\ref{rep of units}$(i)$. Hence $p\Z^{\times}_p \not \rightarrow L$ by Lemma~\ref{rep of units}$(iii)$, contrary to the assumption that $L$ is $\Z_p$-universal.

$r_1 =2$: In this case, $ \Z^{\times}_p \rightarrow L_{(0)} $ and $ p\Z^{\times}_p \rightarrow L_{(1)}$ by Lemma~\ref{unimod}$(i)$. So $L_{(0)} \perp L_{(1)} $ is $\Z_p$-universal and hence $L$ is primitively $\Z_p$-universal by Lemma \ref{main argument} since $\rk{L} \geq 5$.

$r_1 \geq 3$: By Lemma~\ref{unimod}$(ii)$, $p\Z_p \prepd L_{(1)}$. Since $ \Z_p \rightarrow  L_{(0)}$ by Lemma~\ref{unimod}$(i)$, it follows that $L$ is primitively $\Z_p$-universal. 
\end{proof}

\begin{remark}
The conclusion of Proposition~\ref{non dyadic} does not hold when $\rk{L} = 4$, as seen by Example \ref{3.4}.
\end{remark}

\begin{remark}
The arguments in this section would be unchanged if $\Z_p$ were replaced by any non-dyadic local ring.
\end{remark}

\section{Primitively universal $\Z_2$-lattices}

The goal of this section is to prove the $2$-adic analogue of Proposition~\ref{non dyadic}. This result appears as Proposition~\ref{u implies pu p=2}. As the proof is more complicated in the $2$-adic case, this section is divided into three subsections. The first subsection gives a complete analysis of the $\Z_2$-universality and primitive $\Z_2$-universality of modular $\Z_2$-lattices of various ranks. The second subsection is devoted to identifying $\Z_2$-lattices that are $\Z_2$-universal, which will play a key role in the proof of Proposition~\ref{u implies pu p=2} presented in the third subsection.

Before proceeding, we make several observations and establish some terminology that will be in effect throughout the section. Note first that in order to prove that an integral $\Z_2$-lattice is $\Z_2$-universal, it suffices to show that it represents all units and elements of order one (i.e., elements of the set $2\Z_2^{\times}$). For if $\Z_2^{\times}\repd L$ (or $2\Z_2^{\times}\repd L$), then $L$ represents all elements of $\Z_2$ of odd (or even) order. We will refer to a set of elements of $\Z_2$ as being {\em independent} if they are in distinct squareclasses. So, in order to prove that an integral $\Z_2$-lattice is $\Z_2$-universal, it suffices to show that $L$ represents a set of four independent units and a set of four independent elements of $2\Z_2^{\times}$. Throughout this section, the letters $\ep$ or $\ep_i$ will always denote elements of $\Z_2^{\times}$.

\subsection{Modular $\Z_2$-lattices}

In this subsection we determine which modular $\Z_2$-lattices are either $\Z_2$-universal or primitively $\Z_2$-universal. In particular, such a $\Z_2$-lattice $L$ must have $\nm L=\Z_2$ and hence must be either an improper $\frac{1}{2}\Z_2$-modular lattice or a proper unimodular lattice.

\begin{proposition}\label{improper uni}

Let $L$ be an improper $\frac{1}{2}\Z_2$-modular $\Z_2$-lattice of rank $n$.

i) If $n>2$, then $L$ is primitively $\Z_2$-universal.

ii) If $n=2$, then $L$ is primitively $\Z_2$-universal if and only if $L$ is $\Z_2$-universal.

iii) If $n=2$ and $L$ is anisotropic, then $q^*(L)=\Z_2^{\times}$ and $q(L)\cap 2\Z_2^{\times}=\emptyset$.
\end{proposition}

\begin{proof}
These results follow immediately from (\ref{improper unimodular}) and Examples~\ref{3.1} and \ref{3.2}. 
\end{proof}

\begin{proposition}\label{uni rk 2}
Let $L$ be a proper unimodular $\Z_2$-lattice of rank 2. Then $L$ is not $\Z_2$-universal. Moreover,

i) if $dL\equiv 3\modfour$, then $\Z_2^{\times}\rightarrow L$ and $q(L)\cap 2\Z_2^{\times}=\emptyset$;

ii) if $dL\equiv 1\modfour$, then $q(L) \cap \Z_2^{\times}\neq\emptyset$ and \[q(L) \cap \Z_2^{\times}
= \{\alpha \in \Z_2^{\times} : \alpha \equiv \varepsilon\modfour\}.\] for any $\varepsilon \in q(L) \cap \Z_2^{\times}$
\end{proposition}

\begin{proof}
$i)$ Let $L\cong \qf{\ep_1,\ep_2}$. Then it can be verified that $\ep_1,\ep_2,\ep_1+4\ep_2,\ep_2+4\ep_1$ form an independent set of four units represented by $L$. For example, suppose that $\ep_2\equiv \ep_1+4\ep_2 \modeight$. Then $-3\ep_2\equiv \ep_1 \modeight$, and it would follow that $\ep_1\ep_2\equiv 5 \modeight$, contrary to the assumption that $dL\equiv 3\modfour$. The other verifications are similar. For the final assertion, it suffices to note that for any $x,y\in \Z_2^{\times}$, $\ep_1x^2+\ep_2y^2\equiv \ep_1+\ep_2\modeight$ and $\ep_1+\ep_2\equiv 0\modfour$ since $\ep_2\equiv -\ep_1 \modfour$ in this case.

$ii)$ In this case, the underlying space $V$ is anisotropic since $dV\neq -1$. Since $L$ is a $\Z_2$-maximal lattice on $V$, it follows from \cite[Theorem 91:1]{OM} that $q(L)=q(V)\cap \Z_2$. Let $L\cong \qf{\ep_1,\ep_2}$. For any $\alpha \in \Z_2^{\times}$, by comparing the Hasse invariants of the $\Q_2$-spaces $\qf{\ep_1,\ep_2}$ and $\qf{\alpha, \ep_1\ep_2}$, we see that $\alpha\repd V$ if and only if $(\alpha,-\ep_1\ep_2)_2=(\ep_1,\ep_2)_2$. Since $-\ep_1\ep_2\equiv 3\modfour$, the value of the symbol $(\alpha,-\ep_1\ep_2)_2$ is determined by the congruence of $\alpha$ modulo $4\Z_2$, which verifies the assertion. 
\end{proof}

\begin{corollary}\label{split binary}
Let $L\cong M\perp K$, with $M$ proper unimodular of rank 2 and $\sca K\subseteq 2\Z_2$. If $L$ is $\Z_2$-universal, then $\nm K = 2\Z_2$.
\end{corollary}

\begin{proof} If $dM\equiv 3\modfour$ and $\nm K \subseteq 4\Z_2$, then it follows from Proposition~\ref{uni rk 2}$(i)$ that $q(L)\cap 2\Z_2^{\times}=\emptyset$, and so $L$ is not $\Z_2$-universal. So consider the case $dM\equiv 1\modfour$ and let $\ep\in q(M) \cap \Z_2^{\times}$. Suppose that $\nm K \subseteq 4\Z_2$ and $\lambda$ is any unit represented by $L$. Then $\lambda =\mu +\kappa$, where $\mu\in q(M) \cap \Z_2^{\times}$ and $\kappa \in q(K)\subseteq 4\Z_2$. By Proposition~\ref{uni rk 2}$(ii)$, $\lambda \equiv \ep \modfour$. Thus, $\Z_2^{\times}\nrepd L$ and $L$ is not $\Z_2$-universal. 
\end{proof}

\begin{proposition}\label{uni rk 3}
Let $L$ be a unimodular $\Z_2$-lattice of rank 3. Then the following are equivalent:

i) $L$ is $\Z_2$-universal;

ii) $L$ is primitively $\Z_2$-universal;

iii) $L\cong \qf{\ep_1,\ep_2,\ep_3}$ and there exist $i,j \in \{1,2,3\}$ such that $\ep_i\equiv -\ep_j\modfour$.
\end{proposition}

\begin{proof}
$ii)\implies i)$ is clear, and $iii)\implies ii)$ follows from Proposition~\ref{uni rk 2}$(i)$ and Lemma~\ref{split by unit}. So it remains to prove $i)\implies iii)$. For this, suppose that $iii)$ is not true. So $L\cong \qf{\ep_1,\ep_2,\ep_3}$ with $\ep_i\equiv \ep_j\modfour$ for all $1 \leq i,j\leq 3$. Let $\lambda \in q(L)\cap \Z_2^{\times}$. So $\lambda = \sum_{i=1}^3 a_1^2\ep_i$ with $a_i \in \Z_2$ and either one or all three of the $a_i$'s units. If exactly one $a_i \in \Z_2^{\times}$, then $\lambda \equiv \ep_i\modfour$; if all $a_i \in \Z_2^{\times}$, then $\lambda \equiv \ep_1+\ep_2+\ep_3\modeight$. Hence, $L$ can represent at most three squareclasses of units and so $L$ is not $\Z_2$-universal. This completes the proof. 
\end{proof}

\begin{proposition}\label{uni rk 4}
Let $L$ be a proper unimodular $\Z_2$-lattice of rank 4. Then

i) $L$ is $\Z_2$-universal;

ii) $L$ is primitively $\Z_2$-universal if and only if $\{4,8\}\subseteq q^*(L)$.
\end{proposition}

\begin{proof}
There is a splitting $L\cong \qf{\ep_1,\ep_2,\ep_3,\ep_4}$. If $\ep_i\equiv -\ep_j\modfour$ for some $i\neq j$, then $L$ is primitively $\Z_2$-universal by Proposition~\ref{uni rk 3} and Lemma~\ref{split by unit}.

$i)$ It remains to show that $L$ is $\Z_2$-universal when $\ep_i\equiv \ep_j\modfour$ for all $1\leq i,j\leq 4$. By scaling $L$ by a unit if necessary, it suffices to consider the possibilities $L\cong \qf{1,1,a,b}$, where $a,b$ equal $1$ or $5$. In these cases, it is routine to show that $L$ represents the four squareclasses of units and the four squareclasses of twice units.

$ii)$ It suffices to prove the "if" statement. So assume that $L$ is not primitively $\Z_2$-universal. So $\ep_i\equiv \ep_j\modfour$ for all $1\leq i,j\leq 4$. Let $\lambda \in q^*(L)$. So $\lambda =\sum_{i=1}^4 a_i^2\ep_i$ with all $a_i\in\Z_2$ and at least one $a_i\in\Z_2^{\times}$. If exactly one or three of the $a_i$'s are in $\Z_2^{\times}$, then $\lambda \in \Z_2^{\times}$; if exactly two of the $a_i$'s are in $\Z_2^{\times}$, then $\lambda \in 2\Z_2^{\times}$, since $\ep_i\equiv \ep_j\modfour$. So if $\lambda \in 4\Z_2$, it must be that $a_i \in \Z_2^{\times}$ for all $i=1,\ldots,4$ and it follows that $\lambda \equiv \ep_1+\cdots +\ep_4 \modeight$. Moreover $\ep_1+\cdots +\ep_4 \in 4\Z_2$ since $\ep_i\equiv \ep_j\modfour$. If $\ep_1+\cdots +\ep_4 \in 4\Z_2^{\times}$, then $8\not\prepd L$. If $\ep_1+\cdots +\ep_4 \in 8\Z_2$, then $4\not\prepd L$. This completes the proof. 
\end{proof}

\begin{proposition}\label{uni rk 5}
Let $L$ be a unimodular $\Z_2$-lattice of rank exceeding 4. Then $L$ is primitively $\Z_2$-universal.
\end{proposition}

\begin{proof}
Follows from Proposition~\ref{uni rk 4}$(i)$ and Lemma~\ref{split by unit}. 
\end{proof}

\begin{remark}
Corollary 2 of \cite{nB} follows immediately from Lemma~\ref{unimod}$(ii)$, Proposition~\ref{uni rk 4}$(ii)$ and Proposition~\ref{uni rk 5}.
\end{remark}

\subsection{Some $\Z_2$-universal lattices}

The goal of this subsection is to build up an inventory of integral $\Z_2$-lattices that are $\Z_2$-universal. We begin by identifying lattices that represent all units or twice units.

\begin{lemma}\label{reps units}
All $\Z_2$-lattices of the following types represent $\Z_2^{\times}$:

i) $\qf{\ep_1,2\ep_2,\lambda\ep_3}$, where $\lambda=1$ or $4$;

ii) $\qf{\ep_1,\ep_2,\ep_3,\lambda\ep_4}$, where $\lambda=1$ or $4$;

iii) $\qf{\ep_1,2\ep_2,2\ep_3,2\ep_4}$
\end{lemma}

\begin{proof}
$i)$ It suffices to prove the result for the lattice $L\cong \qf{\ep_1,2\ep_2,4\ep_3}$. In this case, it can be routinely verified that $\ep_1, \ep_1+2\ep_2, \ep_1+4\ep_3, \ep_1+2\ep_2+4\ep_3$ form an independent set of four units represented by $L$.

$ii)$ It suffices to prove the result for the lattice $L\cong \qf{\ep_1,\ep_2,\ep_3,4\ep_4}$. By Proposition~\ref{uni rk 3}, it suffices to consider the case $\ep_1\equiv \ep_2\equiv \ep_3\modfour$. Then $\ep_2+\ep_3 \in 2\Z_2^{\times}$ and it follows that $\ep_1, \ep_1+\ep_2+\ep_3, \ep_1+4\ep_4, \ep_1+\ep_2+\ep_3+4\ep_4$ form an independent set of four units represented by $L$.

$iii)$ Let $L\cong \qf{\ep_1,2\ep_2,2\ep_3,2\ep_4}$. At least two of $\ep_2,\ep_3,\ep_4$ are congruent modulo $4\Z_2$; without loss of generality, by re-indexing if necessary, we may assume that $\ep_3\equiv \ep_4\modfour$. If also $\ep_2\equiv \ep_3\modfour$, then $\ep_2+\ep_3\in 2\Z_2^{\times}$ and $\ep_3+\ep_4\in 2\Z_2^{\times}$. From this it follows that $\ep_1,\ep_1+2\ep_2, \ep_1+2\ep_2+2\ep_3, \ep_1+2\ep_2+2\ep_3+2\ep_4$ form an independent set of four units represented by $L$. Otherwise, $\ep_2\equiv -\ep_3\modfour$. Then $\ep_1,\ep_1+2\ep_2, \ep_1+2\ep_3, \ep_1+2\ep_3+2\ep_4$ form an independent set of four units represented by $L$. 
\end{proof}

\begin{lemma}\label{reps twice units}
All $\Z_2$-lattices of the following types represent $2\Z_2^{\times}$:

i) $\qf{\ep_1,\ep_2,\ep_3}$;

ii) $\qf{\ep_1,2\ep_2,\lambda\ep_3}$, where $\lambda = 2$ or $8$;

iii) $\qf{\ep_1,2\ep_2,4\ep_3,\lambda\ep_4}$, where $\lambda = 1$ or $4$;

iv) $\widehat{\A}\perp \qf{2\ep_1,\lambda\ep_2}$, where $\lambda =2,4,8$.
\end{lemma}

\begin{proof}
$i)$ By Proposition~\ref{uni rk 3}, it suffices to consider the case $\ep_1\equiv \ep_2\equiv \ep_3\modfour$. Under this condition, we have $\ep_i+\ep_j\in 2\Z_2^{\times}$ for all $1\leq i,j \leq 3$. Then it can be shown that $\ep_1+\ep_2, \ep_1+9\ep_2, \ep_1+\ep_2+4\ep_3, \ep_1+9\ep_2+4\ep_3$ form an independent set of four elements of $2\Z_2^{\times}$ represented by $L$.

$ii)$ It suffices to consider $L\cong \qf{\ep_1,2\ep_2,8\ep_3}$. Let $\lambda \in \Z_2^{\times}$. By Lemma~\ref{reps units}, there exist $a_1,a_2,a_3\in \Z_2$ such that $\lambda = a_2^2\ep_2+2a_1^2\ep_1+4a_3^2\ep_3$. So $2\lambda = (2a_1)^2\ep_1+a_2^2(2\ep_2)+a_3^2(8\ep_3)\in q(L)$.

$iii)$ It suffices to consider $L\cong \qf{\ep_1,2\ep_2,4\ep_3,4\ep_4}$. Let $\lambda \in \Z_2^{\times}$. By Lemma~\ref{reps units}, there exist $a_1,a_2,a_3, a_4\in \Z_2$ such that $\lambda = a_2^2\ep_2+2a_1^2\ep_1+2a_3^2\ep_3+2a_4^2\ep_4$. Then $2\lambda = (2a_1)^2\ep_1+a_2^2(2\ep_2) +a_3^2(4\ep_3)+a_4^2(4\ep_4) \in q(L)$.

$iv)$ Consider first $L\cong \widehat{\A}\perp \qf{2\ep_1,4\ep_2}$. Let $\lambda \in \Z_2^{\times}$. If $\lambda - \ep_1\in 8\Z_2$, then $2\lambda \in 2\ep_1(\Z_2^{\times})^2$ and $2\lambda \in q(L)$. If $\lambda - \ep_1 \in 2\Z_2^{\times}$, then $2\lambda - 2\ep_1 \in 4\Z_2^{\times} \subseteq q(\widehat{A})$ and $2\lambda \in q(L)$. If $\lambda - \ep_1\in 4\Z_2$, then $2\lambda-2\ep_1\in 8\Z_2$. So $2\lambda-2\ep_1-4\ep_2 \in 4\Z_2^{\times} \subseteq q(\widehat{A})$. So in all cases, $2\lambda \in q(L)$. It remains to consider $L\cong \widehat{\A}\perp \qf{2\ep_1,8\ep_2}$. Since $\Z_2 \repd \widehat{A}$, the result in this case follows from $ii)$. 
\end{proof}

Combining results from the preceding two lemmas and results for unimodular lattices from the previous subsection, and applying Lemma~\ref{split by unit} where necessary, we obtain the $\Z_2$-universal lattices in the following lemma. We note that those in $i)$ through $iv)$ can be obtained from \cite[Lemma 1]{P}; however, that lemma is stated without proof and we have chosen to include the arguments here for the sake of completeness.

\begin{lemma}\label{Z2-universal}
All $\Z_2$-lattices of the following types are $\Z_2$-universal.

i) $\qf{\ep_1,\ep_2,\ep_3,\lambda\ep_4}$, where $\lambda =1,2,4$;

ii) $\qf{\ep_1,\ep_2,2\ep_3,\lambda\ep_4}$, where $\lambda =2,4,8$;

iii) $\qf{\ep_1,2\ep_2,2\ep_3,\lambda\ep_4}$, where $\lambda =2,4$;

iv) $\qf{\ep_1,2\ep_2,4\ep_3,\lambda\ep_4}$, where $\lambda =4,8$;

v) $\widehat{\A} \perp \qf{\ep}$;

vi) $\widehat{\A}\perp \qf{2\ep_1,\lambda\ep_2}$, where $\lambda =2,4,8$.
\end{lemma}

For the proof of Proposition~\ref{u implies pu p=2}, it will also be useful to identify several lattices that fail to represent $\Z_2^{\times}$ or $2\Z_2^{\times}$.

\begin{lemma}\label{not rep units}
$\Z_2^{\times}\not \repd \qf{\ep_1,2\ep_2,2\ep_3}$, $2\Z_2^{\times} \nrepd \qf{\ep_1,2\ep_2,4\ep_3}$, and $2\Z_2^{\times} \nrepd \widehat{\A}\perp \qf{2\ep}$.
\end{lemma}

\begin{proof} First consider $L\cong \qf{\ep_1,2\ep_2,2\ep_3}$. If $\lambda \in \Z_2^{\times}\cap q(L)$, then there exist $a_1\in \Z_2^{\times}$ and $a_2,a_3 \in \Z_2$ such that $\lambda = a_1^2\ep_1+2a_2^2\ep_2+2a_3^2\ep_3$. If $a_2,a_3\in 2\Z_2$, then $\lambda \equiv \ep_1\modeight$. If $a_2\in \Z_2^{\times}$ and $a_3\in 2\Z_2$, then $\lambda \equiv \ep_1+2\ep_2\modeight$. If $a_2\in 2\Z_2$ and $a_3 \in \Z_2^{\times}$, then $\lambda \equiv \ep_1+2\ep_3\modeight$. If $a_2,a_3\in \Z_2^{\times}$, then $\lambda \equiv \ep_1+2\ep_2+2\ep_3\modeight$. When $\ep_2\equiv \ep_3\modfour$, it follows that $\ep_1+2\ep_2\equiv \ep_1+2\ep_3\modeight$. Otherwise $\ep_2\equiv -\ep_3\modfour$, and $\ep_1\equiv \ep_1+2\ep_2+2\ep_3\modfour$. So in either case, $L$ represents at most three squareclasses of units. Hence, $\Z_2^{\times} \not\repd L$.

Next consider $L\cong \qf{\ep_1,2\ep_2,4\ep_3}$. If $\lambda\in \Z_2^{\times}$ is such that $2\lambda \in q(L)$, then there exist $a_1=2b_1\in 2\Z_2$ and $a_2,a_3\in \Z_2$ such that $2\lambda = a_1^2\ep_1+2a_2^2\ep_2+4a_3^2\ep_3$. From this it follows that $\lambda = 2b_1^2\ep_2+a_2^2\ep_2+2a_3^2\ep_3$. So $2\Z_2^{\times}\repd L$ would imply that $\Z_2^{\times}\repd \qf{\ep_2,2\ep_1,2\ep_3}$, which we have just shown to be impossible.

Finally consider $L\cong \widehat{\A}\perp \qf{2\ep}$. Suppose $2\ep +8 \repd L$. Then there exist $v\in \widehat{\A}$ and $\mu \in \Z_2$ such that $2\ep +8 = q(v)+2\mu^2\ep$. It must be that $v \in 2\widehat{\A}$, since otherwise $q(v)\in \Z_2^{\times}$ and hence $q(v)+2\mu^2\ep \in \Z_2^{\times}$. Also, $\mu \in \Z_2^{\times}$, since otherwise $q(v)+2\mu^2\ep \in 4\Z_2$. So $\mu^2\equiv 1 \modeight$; that is, there exists $\xi \in \Z_2$ such that $1-\mu^2=8\xi$. Thus, \[q(v)-8 = 2\ep(1-\mu^2)=16\ep\xi.\] But $\ord_2q(v)$ is even, so that $\ord_2(q(v)-8) = 2$ or $3$, a contradiction. So $2\ep +8 \nrepd L$, and the assertion is proved. 
\end{proof}

\subsection{Primitively $\Z_2$-universal lattices of rank exceeding 4}

We are now ready to state and prove the main result of this section.

\begin{proposition}\label{u implies pu p=2}
Let $L$ be an integral $\Z_2$-lattice of rank $n\geq 5$. If $L$ is $\Z_2$-universal, then $L$ is primitively $\Z_2$-universal.
\end{proposition}

\begin{proof} Since $L$ is $\Z_2$-universal, it must be that $\nm L=\Z_2$ and so $\sca L=\Z_2$ or $\frac{1}{2} \Z_2$. We first consider the case when $\sca L= \frac{1}{2} \Z_2$. So $L_{(-1)}\neq 0$ and $L\cong L_{(-1)}\perp K$, where $\sca K \subseteq \Z_2$. If $r_{-1}>2$, or $r_{-1}=2$ and $L_{(-1)}$ is isotropic, then $L$ is split by $\widehat{\Hy}$ and it follows that $L$ is primitively $\Z_2$-universal. So we need only consider further those lattices $L$ for which there is a splitting of the type $L\cong \widehat{\A}\perp K$, where $\sca K \subseteq \Z_2$. If $\nm K\subseteq 4\Z_2$, then $\widehat{\A}\perp K$ cannot represent any element of $2\Z_2^{\times}$ and is thus not $\Z_2$-universal. So $2\Z_2\subseteq \nm K$, and it follows that $\sca K = \Z_2$ or $2\Z_2$. If $\sca K=\Z_2$, then $L$ primitively $\Z_2$-universal follows from Example 3.5, Lemma \ref{Z2-universal}$(v)$ and Lemma~\ref{main argument}. So we are left to further consider only those lattices for which $\sca K= \nm K = 2\Z_2$. So \[L\cong \widehat{\A}\perp K\cong \widehat{\A}\perp \langle 2\ep\rangle \perp K', \text{ with } \sca K'\subseteq 2\Z_2.\] By Lemma \ref{not rep units}, the sublattice $\widehat{\A}\perp \qf{ 2\ep}$ does not represent all elements of $2\Z_2^{\times}$. Since $L$ is $\Z_2$-universal, it follows from Lemma~\ref{rep of units} that $\nm K'\supseteq 8\Z_2$. Hence, \[8\Z_2\subseteq \nm K' \subseteq \sca K' \subseteq 2\Z_2.\] If $\nm K'=\sca K'=2^t\Z_2$ for $t=1,2,3$, then $L$ is split by a sublattice $\widehat{\A}\perp \qf{2\ep_1,2^t\ep_2}$. All such lattices are $\Z_2$-universal by Lemma \ref{Z2-universal}$(v)$, and it follows from Lemma~\ref{main argument} that $L$ is primitively $\Z_2$-universal, since $n\geq 5$. Finally, consider the case when $\nm K'=8\Z_2=2\sca K'$. Then $L$ is split by $\widehat{\A}\perp \qf{2\ep}\perp \calP$, where $\calP\cong \left(\begin{smallmatrix} 8 & 4\\ 4 & 8\end{smallmatrix}\right)$ or $\left(\begin{smallmatrix} 0 & 4\\ 4 & 0 \end{smallmatrix}\right)$. Note first that $8\Z_2^{\times}\prepd \calP$. If $\lambda \in 4\Z_2^{\times}$, then for any $v\prin \calP$, $\lambda-q(v)\in 4\Z_2^{\times}\repd \widehat{\A}$; hence, $\lambda \prepd \widehat{\A}\perp \calP$. If $\lambda \in 16\Z_2$, then $\lambda-8\ep \in 8\Z_2^{\times}\prepd\calP$; hence, $\lambda \prepd \qf{2\ep}\perp\calP$. This completes the argument in the case that $\sca L= \frac{1}{2} \Z_2$.

\smallskip
Now we consider the case when $\sca L=\Z_2$. Then $r_{-1}=0$ and $r_0>0$. We break down the argument according to the size of $r_0$.

\smallskip
$r_0\geq 4$: $L$ is split by a proper unimodular sublattice of rank 4, which is $\Z_2$-universal by Proposition~\ref{uni rk 4}. Since $n\geq 5$, it follows from Lemma~\ref{main argument} that $L$ is primitively $\Z_2$-universal.

\smallskip
$r_0=3$: $L\cong \qf{\ep_1,\ep_2,\ep_3}\perp K$, where $\sca K \subseteq 2\Z_2$. Since $2\Z_2^{\times}\repd \qf{\ep_1,\ep_2,\ep_3}$ by Lemma \ref{reps twice units}, we need to consider only the case when $\Z_2^{\times} \nrepd \qf{\ep_1,\ep_2,\ep_3}$, since otherwise $\qf{\ep_1,\ep_2,\ep_3}$ is $\Z_2$-universal and $L$ is primitively $\Z_2$-universal by Lemma \ref{main argument}. Since $L$ is $\Z_2$-universal, we must then have $\nm K \supseteq 4\Z_2$, by Lemma~\ref{rep of units}. So we have \[4\Z_2 \subseteq \nm K \subseteq \sca K \subseteq 2\Z_2.\] If $\nm K = \sca K = 2^t\Z_2$, for $t=1,2$, then $L$ is split by $\qf{\ep_1,\ep_2,\ep_3,2^t\ep_4}$ which is $\Z_2$-universal. Otherwise, $\nm K = 4\Z_2 = 2\sca K$, in which case $L$ is split by $\qf{\ep_1,\ep_2,\ep_3} \perp \calP$, where $\calP\cong \left(\begin{smallmatrix} 4 & 2\\ 2 & 4\end{smallmatrix}\right)$ or $\left(\begin{smallmatrix} 0 & 2\\ 2 & 0\end{smallmatrix}\right)$. Then $4\Z_2^{\times} \prepd \calP$. If $\lambda \in 8\Z_2$, then $\lambda-4\ep_1 \in 4\Z_2^{\times}$ and so $\lambda \prepd \qf{\ep_1}\perp \calP$. That completes this subcase.

\smallskip
$r_0=2$: $L\cong \qf{\ep_1,\ep_2}\perp K$, with $\sca K \subseteq 2\Z_2$. Since $L$ is $\Z_2$-universal, we have $\nm K = \sca K =2\Z_2$ by Corollary~\ref{split binary}. So $r_1>0$ and $L_{(1)}$ is proper. So if $r_1\geq 2$, $L$ is split by a sublattice $\qf{\ep_1,\ep_2,2\ep_3,2\ep_4}$, which is $\Z_2$-universal by Lemma~\ref{Z2-universal}$(ii)$ and the conclusion follows. So we further consider the case $r_1=1$; that is, \[L\cong \qf{\ep_1,\ep_2,2\ep_3}\perp K',\,\text{ with }\sca K' \subseteq 4\Z_2.\] We further assume that $2\Z_2^{\times}\not\rightarrow \qf{\ep_1,\ep_2,2\ep_3}$, since otherwise $\qf{\ep_1,\ep_2,2\ep_3}$ is $\Z_2$-universal and there is nothing to prove. So, since $L$ is $\Z_2$-universal, we must have $\nm K' \not\subseteq 16 \Z_2$. So \[8\Z_2 \subseteq \nm K' \subseteq \sca K' \subseteq 4\Z_2.\] If $\nm K'=\sca K' = 2^t\Z_2$, with $t=2,3$, then $L$ is split by a $\Z_2$-universal lattice of the type $\qf{\ep_1,\ep_2,2\ep_3,2^t\ep_4}$. Otherwise, $\nm K'=8\Z_2 = 2\sca K'$ and $L$ is split by a lattice of the type $\qf{\ep_1,\ep_2,2\ep_3}\perp \calP$, with $\calP \cong \left(\begin{smallmatrix} 8 & 4\\ 4 & 8\end{smallmatrix}\right)$ or $\left(\begin{smallmatrix} 0 & 4\\ 4 & 0 \end{smallmatrix}\right)$. So $8\Z_2^{\times}\prepd \calP$; let $v\prin \calP$ such that $q(v)=8$. If $\lambda \in 4\Z_2^{\times}$, then $\lambda-q(v) \in 4\Z_2^{\times} \repd \qf{\ep_1,\ep_2,2\ep_3}$ (since $\Z_2^{\times}\repd \qf{\ep_1,\ep_2,2\ep_3}$), and $\lambda \prepd L$. Finally if $\lambda \in 16 \Z_2$, then $\lambda - 2^2\cdot 2\ep_3 \in 8\Z_2^{\times}$ and so $\lambda - 2^2\cdot 2\ep_3 \prepd \calP$ and $\lambda\prepd \qf{2\ep_3}\perp \calP$.

\smallskip
$r_0=1$: $L\cong \qf{\ep_1}\perp K$, with $\sca K \subseteq 2\Z_2$. Since $L$ is $\Z_2$-universal, $\nm K = \sca K = 2\Z_2$ (since otherwise $q(L)\cap 2\Z_2^{\times}=\emptyset$). So $r_1>0$ and $L_{(1)}$ is proper. We consider the various possibilities for $r_1$. If $r_1\geq 3$, then $L$ is split by $\qf{\ep_1,2\ep_2,2\ep_3,2\ep_4}$ which is  $\Z_2$-universal. If $r_1=2$, then $L\cong \qf{\ep_1,2\ep_2,2\ep_3}\perp K'$, with $\sca K'\subseteq 4\Z_2$. Since $L$ is $\Z_2$-universal and $\Z_2^{\times}\nrepd \qf{\ep_1,2\ep_2,2\ep_3}$ by Lemma \ref{not rep units}, it follows from Lemma \ref{rep of units} that $\nm K'=\sca K' =4\Z_2$. So $r_2>0$ and $L_{(2)}$ is proper, so $L$ is split by $\qf{\ep_1,2\ep_2,2\ep_3,4\ep_4}$, which is  $\Z_2$-universal. The only remaining case is when $r_1=1$. Then $L\cong \qf{\ep_1,2\ep_2}\perp K$, with $\sca K \subseteq 4\Z_2$. Since $L$ is universal and $\Z_2^{\times}\nrepd \qf{\ep_1,2\ep_2}$ by Lemma~\ref{not rep units}, it follows from Lemma~\ref{rep of units} that $\nm K = \sca K = 4\Z_2$. So $r_2>0$ and $L_{(2)}$ is proper. If $r_2\geq 2$, then $L$ is split by $\qf{\ep_1,2\ep_2,4\ep_3,4\ep_4}$, which is $\Z_2$-universal. Finally, it remains to consider the subcase when $r_2=1$. Then $L\cong \qf{\ep_1,2\ep_2,4\ep_3} \perp K''$, with $\sca K''\subseteq 8\Z_2$. Since $L$ is universal and $2\Z_2^{\times}\nrepd \qf{\ep_1,2\ep_2,4\ep_3}$ by Lemma~\ref{not rep units}, it follows that Lemma \ref{rep of units} that $\nm K'=\sca K' =8\Z_2$. So $r_3>0$ and $L_{(3)}$ is proper, and $L$ is split by $\qf{\ep_1,2\ep_2,4\ep_3,8\ep_4}$, which is $\Z_2$-universal. 
\end{proof}

\section{Proofs of theorems}

We conclude the paper by supplying proofs of the theorems stated in the Introduction. As the theorems are stated there in the traditional language of quadratic forms, we will first review the connections between quadratic forms and quadratic lattices, as used in Sections 2 through 5. A nondegenerate integral quadratic form $f=f(X_1,\ldots,X_n)$ of rank $n$ can be written as \begin{equation}\label{form}f=\sum_{1\leq i,j \leq n}a_{ij}X_iX_j, \text{ where } a_{ij}=a_{ji}, a_{ii}\in \Z, 2a_{ij}\in \Z \text{ for } i\neq j.\end{equation}  Let $M_f$ denote the symmetric matrix $(a_{ij})$, and associate to $f$ a quadratic $\Z$-lattice $L$ and basis $\mathcal B$ for $L$ such that the Gram matrix of $L$ with respect to $\mathcal B$ is $M_f$.  An integer is (primitively) represented by the form $f$ if and only if it is (primitively) represented by the associated lattice $L$. Our main tool for proving Theorems 1.1 through 1.3 is the following result that relates the positive integers primitively represented by a positive definite quadratic $\Z$-lattice to those integers that are primitively represented by all local completions $L_p$ of the lattice: {\em Let $L$ be a positive definite integral $\Z$-lattice of rank $n\geq 4$. Then there is an integer $N$ with the following property: If $a\geq N$ is an integer that is primitively represented by $L_p$ for all primes $p$, then $a$ is primitively represented by $L$.} (see, e.g., \cite[Theorem 1.6, page 204]{Ca}). In particular, such a lattice $L$ is almost primitively universal if and only if $L_p$ is primitively $\Z_p$-universal for all primes $p$. From this, the proofs of Theorems~\ref{Theorem 1} and \ref{Theorem 2} are now immediate.
\bigskip

\noindent {\it Proof of Theorem~\ref{Theorem 1}}\,
Follows from Corollary 3.10. \qed
\bigskip

\noindent {\it Proof of Theorem~\ref{Theorem 2}}\,
Follows from Propositions 4.2 and 5.12. \qed
\bigskip

The form (\ref{form}) is said to be {\em classically integral} if $a_{ij}\in \Z$ for all $i,j$. In this case, the discriminant $df=\det\,L$ is an integer, and \[\ord_p\vol L_p=\ord_p\vol L=\ord_pdL=\ord_pdf\] for all primes $p$. In particular, for a classically integral form $f$ and positive integer $t$, \[p^t \mid df \text{   if and only if   } \vol L_p \subseteq p^t\Z_p.\]
\smallskip

\noindent {\it Proof of Theorem~\ref{Theorem 3}}\, Let $L$ be a positive definite $\Z$-lattice for which $\sca L \subseteq \Z$, $\rk{L} = n\geq 4$ and, for all primes $p$, $p^{n-2}\nmid dL$. Moreover, it is assumed that $L$ represents an odd integer, and that $dL$ is even when $n=4$. If $\sca L_p \subseteq p\Z_p$ for some prime $p$, then $\vol L_p \subseteq p^n\Z_p$ and it would follow that $p^n\mid dL$, contrary to assumption. Hence, $\sca L = \Z$. When $p=2$, the assumption that $L$ represents an odd integer guarantees that $\nm L=\Z$ as well. So for each prime $p$, $L_p$ has a splitting of the type $L_p \cong L_{(0)}\perp K$, where $L_{(0)}$ is diagonalizable and $\sca K\subseteq p\Z_p$ or $K=0$. Since $p^{n-2}\nmid dL$, it follows that $\rk{K}\leq n-3$ and so $r_0 = \rk{L_{(0)}}\geq 3$. So, if $p$ is odd, $L_p$ is primitively $\Z_p$-universal by Lemma 4.1$(ii)$. So we need only consider further the case $p=2$. If $r_0\geq 4$, then $L_2$ is split by $N\cong\qf{\ep_1,\ep_2,\ep_3,\ep_4}$, which is $\Z_2$-universal by Lemma~\ref{Z2-universal}. The assumption that $dL$ is even when $n=4$ rules out the possibility that $L=N$; so $n\geq 5$ and it follows from Theorem~\ref{Theorem 2} that $L_2$ is primitively $\Z_2$-universal. So to complete the proof, we consider the case $r_0=3$. So \[L_2\cong \qf{\ep_1,\ep_2,\ep_3} \perp K, \text{ with }\sca K\subseteq 2\Z_2 \text{ and } \rk{K}=n-3.\] If $\sca K \subseteq 4\Z_2$, then $\vol L_2=\vol K \subseteq 2^{2(n-3)}\Z_2$; but $2(n-3)\geq n-2$ since $n\geq 4$, thus contradicting the assumption that $p^{n-2}\nmid dL$. So $\sca L_2=2\Z_2$. This leaves two possibilities: $\nm K=2\Z_2$ or $\nm K=4\Z_2$. First consider the case $\nm K=4\Z_2$. Then $\rk{K}\geq 2$ (since $\nm K \neq \sca K$) and $n\geq 5$. Since $\nm K=4\Z_2$, there exists $\ep_4 \in \Z_2^{\times}$ such that $4\ep_4 \repd K$. So $L_2$ contains a sublattice $\qf{\ep_1,\ep_2,\ep_3,4\ep_4}$, which is $\Z_2$-universal by Lemma~\ref{Z2-universal}. So $L_2$ is $\Z_2$-universal and hence primitively $\Z_2$-universal by Theorem~\ref{Theorem 2}. Finally, consider the case $\nm K=\sca K=2\Z_2$. Then there exists $\ep_4 \in \Z_2^{\times}$ such that $2\ep_4 \repd K$. So $L_2$ contains a sublattice $N\cong \qf{\ep_1,\ep_2,\ep_3,2\ep_4}$, which is $\Z_2$-universal by Lemma~\ref{Z2-universal}. If $n\geq 5$, it follows from Theorem~\ref{Theorem 2} that $L_2$ is primitively $\Z_2$-universal. If $n=4$, then $L_2=N$. Since $\Z_2^{\times} \repd \qf{\ep_2,\ep_3,2\ep_4}$ by Lemma~\ref{reps units}, it follows from Lemma~\ref{split by unit} that $L_2$ is primitively $\Z_2$-universal. This completes the proof. \qed
\bigskip

The proof of Theorem~\ref{Theorem 4} relies on several fundamental results from spinor genus theory. For general background on the spinor genus, the reader is referred, e.g., to \cite[\S102A]{OM} or \cite[Chapter 11]{Ca}. For a $\Z$-lattice $L$, the genus, spinor genus and isometry class of $L$ will be denoted by $\text{gen}\,L$, $\text{spn}\,L$ and $\text{cls}\,L$, respectively. If $\mathcal S$ denotes one of the objects $\text{gen}\,L$, $\text{spn}\,L$ or $\text{cls}\,L$, the notations $a\repd \mathcal S$ or $a\prepd \mathcal S$ will mean that there exists a lattice $K\in\mathcal S$ such that $a\repd K$ or $a\prepd K$, respectively.
\bigskip

\noindent {\it Proof of Theorem~\ref{Theorem 4}}\, Let $L$ be an indefinite integral $\Z$-lattice of rank $n\geq 5$ such that every integer is represented by $\text{gen}\,L$. Then $L_p$ is $\Z_p$-universal for all primes $p$. So, by Propositions~\ref{non dyadic} and \ref{u implies pu p=2}, $L_p$ is primitively $\Z_p$-universal. Let $0\neq a \in \Z$. Then $a\prepd L_p$ for all $p$, and it follows as in \cite[Example 102:5]{OM} that $a\prepd \text{gen}\,L$. Since $n\geq 4$, it then follows from \cite[Theorem 7.1, page 227]{Ca} that $a\prepd \text{spn}\,L$. Since $L$ is indefinite and $n\geq 3$, $\text{spn}\,L = \text{cls}\,L$ by \cite[Theorem 104:5]{OM}. Hence $a\prepd \text{cls}\,L$ and we conclude that $a\prepd L$, as desired.\qed

\end{document}